\documentclass[12pt,reqno]{amsart}
\usepackage{cite}
\usepackage{amssymb,latexsym,amsmath,amsfonts}
\usepackage{latexsym}
\usepackage[mathscr]{eucal}
= 8.9in
\theoremstyle{plain}
\newtheorem{theorem}{\textbf{Theorem}}[section]
\newtheorem{lemma}[theorem]{\textbf{Lemma}}

\newtheorem{remark}[theorem]{\textbf{Remark}}

\newtheorem{definition}[theorem]{Definition}

\newtheorem{example}[theorem]{\textbf{Example}}

\numberwithin{equation}{section}
\begin{document}

\title[ A new contraction principle on the perimeters of triangles and related results ]
{  A new contraction principle on the perimeters of triangles and related results}
\author[ Tanusri Senapati ]%
{ Tanusri Senapati  }


\address{ Tanusri Senapati,
                    Department of Mathematics,
                    Guskara Mahavidyalaya,
                    West Bengal,
                    India.}
                    \email{tanusri@gushkaramahavidyalaya.ac.in}



\begin{abstract}
In this article, we introduce a new type of mapping contracting perimeters of triangles in a complete metric space and present related fixed point theorem. We study the metric completeness property of the underlying space in terms of fixed point of our newly introduced mapping. In support of our result, we present several examples.
\end{abstract}

\maketitle

\section{\bf Introduction}
\baselineskip .55 cm
Very recently, E. Petrov \cite{pet} presented fixed point theorems for mapping contracting perimeters of triangles in complete metric spaces. To study this type of mappings contracting perimeter of triangles we need at least three distinct points of the underlying spaces. Also, the author of \cite{pet} proved that, like Banach contraction \cite{ban}, the mapping contracting perimeter of  triangles in a metric space is continuous at every point in the underlying metric space. Inspiring by the work of Petrov \cite{pet} and $\acute{C}$iri$\acute{c}$ \cite{cir1,cir2}, in this article, we introduce a new type of mapping contracting perimeters of triangles in a metric space. We observe that our mapping contracting perimeters of triangles is not necessarily continuous at every point of the underlying space. Before we move into details, we first state some basic results related to this work.

\begin{definition}
	Let ($\mathbb{X},d)$ be a metric space with at least three points and let $T:\mathbb{X } \to \mathbb{X}$ be a mapping. Suppose for every three distinct $x,y,z\in \mathbb{X}$, $T$ satisfies the following condition:
	$$d(Tx,Ty) +d(Ty,Tz) +d(Tz,Tx) \leq k M\big (d(x',y')+d(y', z') +d(z', x')\big )$$ where $k\in (0,1)$ and $M\big(d(x', y') +d(y', z') +d(z', x') \big)$ is defined as
	$$\max\big\{d(x', y') +d(y', z') +d(z', x'): x',y',z'\in \{x,y,z,Tx,Ty,Tz\}~
	\mbox{and}~ x'\neq y'\neq z'\big\}.$$
	We call this newly introduced map as a weak contraction map on the perimeters of triangles in $\mathbb{X}$.
\end{definition}
Throughout this article, we always consider $(\mathbb{X},d)$ is a metric space with $|\mathbb{X}|\geq 3$. For any self mapping $T$ on $\mathbb{X}$, the orbit of $T$ at any point $x\in \mathbb{X}$ is  denoted by  $O(x,\infty)=\{T^ix:i\in  \mathbb{N}\cup \{0\}\}$. Also for any $n\in \mathbb{N}$, $O(x,n)=\{T^ix:i\in  \{0,1,2\dots,n\}\}$. Let $n\geq 2$. Then for  any three pairwise distinct $i,j,k$; we define: 
$$p(O(x,n))=\max \{d(T^ix,T^jx)+d(T^jx,T^kx)+d(T^kx,T^ix): i,j,k\in \{0,1,\dots,n\}\}$$ and
$$p(O(x,\infty))=\sup \{d(T^ix,T^jx)+d(T^jx,T^kx)+d(T^kx,T^ix): i,j,k\in \mathbb{N}\cup \{0\} \}.$$ 

Now using the condition of weak contraction on the perimeters of triangles in $\mathbb{X}$ we derive the following results.
\begin{lemma}
	Let $T$ be a weak contraction on the perimeters of triangles in $\mathbb{X}$ and let  $n$ be any positive integer $\geq 2$. Then for any $x\in \mathbb{X}$ and distinct $i,j,k\in\{0,1,2\dots,n\}$
	$$d(T^ix,T^jx)+d(T^jx,T^kx)+d(T^kx,T^ix)\leq kp(O(x,n)).$$
	\begin{proof}
		The proof follows directly from the definition of $p(O(x,n))$. 
	\end{proof}
\end{lemma} 
\begin{lemma}
	Let $T$ be a weak contraction on the perimeters of triangles in $\mathbb{X}$ and let $n$ be any positive integer $\geq 2$. Then for any $x\in \mathbb{X}$,
	$$p(O(x,n))=d(x,T^ix)+d(T^ix,T^jx)+d(T^jx,x),$$
	for some distinct $i,j\in \{1,2,3\dots,n\}$.
	\begin{proof}
		If possible, let $p(O(x,n))=d(T^kx,T^ix)+d(T^ix,T^jx)+d(T^jx,T^kx)$ for some pairwise distinct $i,j,k\in\{1,2,3\dots,n\}$. Since $T$ is a weak contraction on the perimeters of triangles in $\mathbb{X}$, so we must have 
		\begin{align*}    
		p(O(x,n))= & d(T^kx,T^ix)+d(T^ix,T^jx)+d(T^jx,T^kx)\\
		\leq & k \max \big{\{}d(x', y') +d(y', z') +d(z', x'): x',y',z'\in \{T^ix, T^jx,\\
		& T^kx,T^{i-1}x, T^{j-1}x,T^{k-1}x\}; x'\neq y'\neq z'\big{\}} \\
		\leq & k p(O(x,n)).
		\end{align*}
		This is a contradiction and hence the result follows. 
	\end{proof}
\end{lemma} 
\begin{lemma}
	Let $T$ be a weak contraction on the perimeters of triangles in $\mathbb{X}$. Then  $p(O(x,\infty))$ is bounded.
	\begin{proof}  
		Note that for each $n\geq 2$, $p(O(x,n))\leq p(O(x,n+1))$. Therefore, $p(O(x,\infty))=\displaystyle \sup_{n \to \infty} p(O(x,n))$. There exist some distinct $i,j\in \{1,2,3\dots ,n\}$  such that 
		\begin{align*}
		p(O(x,n)) & = d(x,T^ix)+d(T^ix,T^jx)_+d(T^jx,x)\\
		& \leq d(x,Tx)+d(Tx,T^ix)++d(T^ix,T^jx)+d(T^jx,Tx)+d(Tx,x)\\
		& \leq 2 d(x,Tx)+ k p(O(x,n))\\ 
		&\leq  \frac{2}{1-k} d(x,Tx). 
		\end{align*}
		This is true for every $n\geq 2$. Therefore we must have that $p(O(x,\infty))\leq \frac{2}{1-k} d(x,Tx)$. 
	\end{proof}
\end{lemma} 

\section{\bf Main Results}

We start this section by presenting an example which shows that every map contracting perimeters of triangles in $\mathbb{X}$ must be a weak contraction map on the perimeters of triangles but the converse is not true. 
\begin{example}
	Let $\mathbb{X}=\{1,2,3,4\}$ and $d(x,y)=|x-y|$ for all $x,y\in \mathbb{X}$. So $(\mathbb{X},d)$ is a complete metric space. Let $T:\mathbb{X}\to \mathbb{X}$ be defined as $T1=1, Tn=n-1$ otherwise. We show that $T$ is a weak contraction map on the perimeters of triangles in $\mathbb{X}$. Let $x=2,y=3$ and $z=4$. Then 
	$$d(Tx,Ty)+d(Ty,Tz)+d(Tz,Tx)=d(1,2)+d(2,3)+d(3,1)=4$$
	and $$d(x,y)+d(y,z)+d(z,x)=d(2,3)+d(3,4)+d(4,2)=4.$$
	Therefore, $T$ is not a contraction map on the perimeters of triangles in $\mathbb{X}$. Now observe that $\{x,y,z,Tx,Ty,Tz\}= \{1,2,3,4\}$ and $M\big (d(x',y')+d(y',z')+d(z',x')\big )=6$. Considering all other possible values of $x,y,z\in \mathbb{X}$, one can deduce that
	$$d(Tx,Ty)+d(Ty,Tz)+d(Tz,Tx)\leq k M\big (d(x',y')+d(y',z')+d(z',x') \big),$$
	where $x',y',z'\in \{x,y,z,Tx,Ty,Tz\}$ and $k\in (\frac{2}{3},1)$. Thus, $T$ is a weak contraction map on the perimeters of triangles in $\mathbb{X}$.
\end{example}  
\begin{remark}
	Note that the map contracting perimeters of triangles in any metric space is always continuous at every point of the underlying space, see \cite{pet}. However, the weak contraction map is not necessarily continuous at every point.
\end{remark}
In support of our claim, we provide the following example:
\begin{example}
	Let $\mathbb{X}=[0,1]$ and let $d$ be the usual metric on $\mathbb{X}$. Let $T:\mathbb{X} \to \mathbb{X}$ be defined as
	\begin{equation*}
	\begin{split}
	T(\frac{1}{n}) & =\frac{1}{n+1}~ \mbox{for} ~n=1,2,3;\\
	Tx & =0 ~ \mbox{otherwise} .
	\end{split}
	\end{equation*}
	At first we show that $T$ is not a contraction map on perimeters of triangles but a weak contraction map. Set $A=\{1,\frac{1}{2},\frac{1}{3}\}$ and $B=\mathbb{X}\setminus A$. Now whenever we choose $x,y,z\in A $ or $x,y,z\in B$ with $x\neq y\neq z$, then $T$ simply follows the condition of the contraction map on the perimeters of triangles. In other cases, $T$ does not always satisfy the same. For example, we choose $x=1,y=\frac{9}{10}, z=\frac{8}{10}$. Then
	$$d(Tx,Ty) +d(Ty,Tz) +d(Tz,Tx) =1~\mbox{and}~d(x,y) +d(y,z) +d(z,x)=\frac{4}{10}. $$  It follows that there exists no $k\in (0,1)$ such that $d(Tx,Ty) +d(Ty,Tz) +d(Tz,Tx)\leq k(d(x,y) +d(y,z) +d(z,x)) $ holds.  Now if we  choose $x',y',z'\in \{1,\frac{9}{10}, \frac{8}{10},\frac{1}{2},0\}$ with $x'\neq y'\neq z'$ then for any $k\in (\frac{1}{2},1)$ we have 
	$$d(Tx,Ty) +d(Ty,Tz) +d(Tz,Tx)\leq k M\big(d(x',y') +d(y',z') +d(z',x')\big).$$ 
	Observe that one can choose $x,y,z\in \mathbb{X}$ in two possible ways; either $x,y\in A; z\in B$ or $x\in A;y,z\in B$. In every case, $T$ will satisfy the weak contraction condition on the perimeters of triangles in $\mathbb{X}$.  In particular, whenever $x,y\in A; z\in B,$ then
	$$d(Tx,Ty) +d(Ty,Tz) +d(Tz,Tx)\leq k M\big(d(x',y') +d(y',z') +d(z',x')\big)$$ for some 
	$k\in (\frac{2}{3},1)$ and for $x\in A; y,z\in B$
	$$d(Tx,Ty) +d(Ty,Tz) +d(Tz,Tx)\leq k M \big(d(x',y') +d(y',z') +d(z',x')\big)$$ for some 
	$k\in (\frac{3}{4},1).$
	
	Therefore, we can write for any such choice of $x,y,z$ with $x\neq y\neq z$, there exists $k\in (\frac{3}{4},1)$ such that $$d(Tx,Ty) +d(Ty,Tz) +d(Tz,Tx)\leq k M\big(d(x',y') +d(y',z') +d(z',x')\big)$$  holds, where $x',y',z'\in \{x,y,z,Tx,Ty,Tz\}$ and $x'\neq y'\neq z'$.\\
	Next, we show that $T$ is not continuous at every point in $\mathbb{X}$. Let us choose a sequence $(x_n)$ where $x_n=1-\frac{1}{n}$ for all $n\in \mathbf{N}$. Clearly, $x_n\to 1$ but $Tx_n \not\rightarrow T1$ as $n\to \infty$, i.e. $T$ is not continuous at $1$.
\end{example}

Now are in a position to state our main result.
\begin{theorem}
	Let $(\mathbb{X},d)$ be a complete metric space and let $T:\mathbb{X } \to \mathbb{X}$ be a weak contraction map on the perimeters of triangles in $\mathbb{X}$. Suppose  $T^2(x) \neq x$ for all $x\in \mathbb{X}$ with $Tx\neq x$,
	Then for any $x_0\in \mathbb{X}, T^nx_0 \to x^*\in \mathbb{X}$ as $n\to \infty$ and  $x^*$ is a  fixed point of $T$. If $x^*\neq T^nx_0$ for any $n\in \mathbb{N}\cup\{0\}$, then $x^*$ must be unique fixed point of $T$. 
\end{theorem}
\begin{proof}
	Let us consider the iteration $x_n=T^nx_0$ for every $n\in \mathbb{N}$. Assume that $x_i$ is not  a fixed point of $T$ for any $i\in \mathbb{N}\cup \{0\}$.  Then it can be shown that any three consecutive points $x_i,x_{i+1},x_{i+2}$ must be distinct for all $i\in \mathbb{N}\cup \{0\}$. Now,
	$$d(Tx_0,Tx_1) +d(Tx_1,Tx_2) +d(Tx_2,Tx_0) \leq k M\big(d(x',y') +d(y',z') +d(z',x')\big)$$
	where $x',y',z'\in \{x_0,x_1,x_2, Tx_0,Tx_1,Tx_2\}=\{x_0,x_1,x_2, x_3\}$ and $x'\neq y'\neq z'$. Let us set $$\alpha=d(x_0,x_1) +d(x_1,x_2) +d(x_2,x_0),$$ 
	$$\beta=d(x_0,x_1) +d(x_1,x_3) +d(x_3,x_0),$$ $$\gamma=d(x_0,x_2) +d(x_2,x_3) +d(x_3,x_0),$$ $$\lambda=d(x_1,x_2) +d(x_2,x_3) +d(x_3,x_1). $$
	Observe that $$d(Tx_0,Tx_1) +d(Tx_1,Tx_2) +d(Tx_2,Tx_0) \leq k \lambda $$ is not possible. In other cases, we get $$         d(Tx_0,Tx_1) +d(Tx_1,Tx_2) +d(Tx_2,Tx_0) \leq k M $$ where $M=p(O(x,\infty))$.
	
	Again,
	$$ d(Tx_1,Tx_2) +d(Tx_2,Tx_3) +d(Tx_3,Tx_1) \leq k \max\{\alpha_1,\beta_1, \gamma_1,\lambda_1\}$$
	where $$\alpha_1=d(x_1,x_2) +d(x_2,x_3) +d(x_3,x_1),$$ 
	$$\beta_1=d(x_1,x_2) +d(x_2,x_4) +d(x_4,x_1),$$ $$\gamma_1=d(x_1,x_3) +d(x_3,x_4) +d(x_4,x_1),$$
	$$\lambda_1=d(x_2,x_3) +d(x_3,x_4) +d(x_4,x_2). $$
	Observe that $$d(Tx_1,Tx_2) +d(Tx_2,Tx_3) +d(Tx_3,Tx_1) \leq k \lambda_1 $$ is not possible. Moreover, we already have $\alpha_1\leq k M$. In a similar manner, one can deduce that $\beta_1\leq k M$ as well as $\gamma_1\leq kM$. Therefore we obtain $$ d(Tx_1,Tx_2) +d(Tx_2,Tx_3) +d(Tx_3,Tx_1) \leq k^2 M.$$ 
	By following the similar arguments, for any $n\in \mathbb{N}$, we obtain  $$ d(Tx_n,Tx_{n+1}) +d(Tx_{n+1},Tx_{n+2}) +d(Tx_{n+2},Tx_n) \leq k^{n+1} M.$$
	Next, we show that $(x_n)$ is a Cauchy sequence. For any $n,p\in \mathbb{N}$, by using triangle inequalities, we have 
	\begin{align*}
	d(x_n,x_{n+p}) &\leq d(x_n,x_{n+1})+d(x_{n+1},x_{n+2})+...+ d(x_{n+p-1},x_{n+p})\\
	&\leq k^n M+k^{n+1} M+...+ +k^{n+p-1} M\\
	&\leq k^n M(1+k +...+k^
	{p-1})
	\end{align*}
	Therefore, $d(x_n, x_{n+p}) \to 0$ as $n\to \infty$, i.e., $(x_n)$ is a Cauchy sequence. Since $(\mathbb{X},d)$ is a complete metric space, let $x_n\to x^*\in \mathbb{X}$ as $n\to \infty$. We claim that $x^*$ must be a fixed point of $T$. Now, we get
	$$	d(Tx^*,Tx_{n-1})+d(Tx_{n-1}, Tx_n)+d(Tx_n, Tx^*)
	\leq kM\big( d(x', y')+d(y', z')+d(z', x')\big),$$
	where $x',y',z'\in \{x^*, x_{n-1},x_n,Tx^*,Tx_{n-1},Tx_n\}=\{x^*, x_{n-1},x_n,Tx^*,x_{n+1}\}$ with $x'\neq y'\neq z'.$
	Note that one can choose $x',y',z'$ in $\binom{5}{3}-1$ possible ways. Taking $n\to \infty$ in the both sides of the above inequality, one can derive that $d(Tx^*,Tx_{n-1})+d(Tx_{n-1}, Tx_n)+d(Tx_n, Tx^*)=2d(x^*,Tx^*)$  and $M\big( d(x', y')+d(y', z')+d(z', x')\big)= 2d(x^*,Tx^*)$. This implies that $d(x^*,Tx^*)\leq k d(x^*,Tx^*)$ for some $k\in (0,1)$. This is a contradiction. 
	Therefore, $x^*$ must be a fixed point of $T$.\\
	Next,  we consider  $x^*\neq T^nx_0$ for any $n\in \mathbb{N}\cup\{0\}$. We show that $x^*$ is a unique fixed point of $T$. If possible, let $x^{**}$ be another fixed point of $T$. Clearly, $x^{**}\neq T^nx_0$ for any $n\in \mathbb{N}\cup\{0\}$. Also, for any  $n\in \mathbb{N}\cup\{0\}$,  $x^*,x^{**},x_n$ are all distinct. Then,
	$$d(Tx^*,Tx_{n})+d(Tx_{n}, Tx^{**})+d(Tx^{**}, Tx^*)
	\leq kM\big( d(x', y')+d(y', z')+d(z', x')\big),$$
	where $x',y',z'\in \{x^*, x_n, x_{n+1},x^{**}\}$ with $x'\neq y'\neq z'.$ Taking $n\to \infty$ in the the both side of the above inequality, we obtain
	$$d(x^*,x^{**})\leq k d(x^*,x^{**})$$ which is a contradiction. So $x^*$ must be a unique fixed point.
\end{proof}
\begin{remark}
	A weak contraction map on the perimeters of triangles in a metric space may have at most two fixed points.
\end{remark}
\begin{example}
	Let $\mathbb{X}=\{0,1,2,3,4\}$ and $d(x,y)=|x-y|$ for all $x,y\in \mathbb{X}$. So $(\mathbb{X},d)$ is a complete metric space. Let $T:\mathbb{X}\to \mathbb{X}$ be defined as $T0=0,T1=1$ and $ Tn=n-1$ otherwise. One can verify that  except $(x,y,z)=(2,3,4)$, for any three pairwise distinct $x,y,z\in \mathbb{X}$,
	$$d(Tx,Ty)+d(Ty,Tz)+d(Tx,Tx)\leq k d(x,y)+d(y,z)+d(z,x),$$ for some $k\in (\frac{3}{4},1)$.
	For $(x,y,z)=(2,3,4)$,
	$$d(Tx,Ty)+d(Ty,Tz)+d(Tx,Tx)\leq k d(Tx,Ty)+d(Ty,z)+d(z,Tx)$$ where $k\in (\frac{2}{3},1)$.
	Therefore, for any three pairwise distinct $x,y,z\in \mathbb{X}$, 
	$$d(Tx,Ty)+d(Ty,Tz)+d(Tx,Tx)\leq k M\big(d(x',y')+d(y',z')+d(z',x')\big)$$
	where $x',y',z'\in \{Tx,Ty,Tz,x,y,z\}$ are pairwise distinct and $k\in (\frac{3}{4},1)$.
	This shows that $T$ is a weak contraction map on the perimeters of triangles in $\mathbb{X}$. Note that $T$ has two fixed points, namely $0$ and $1$.   
\end{example}
We have already explained that every weak contraction on the perimeters of triangles is not necessarily continuous at every point of the underlying space. However, in the following theorem, we prove that weak contraction map $T$ on the perimeters of triangles in  any metric space $\mathbb{X}$ is continuous at a point $x^*$, where $x^*=\displaystyle\lim_{n\to \infty}T^nx$, for some $x\in \mathbb{X}$.

\begin{theorem}
	Let $T$ be a weak contraction on the perimeters of triangles in $\mathbb{X}$ and let $x\in \mathbb{X}$ such that $T^nx\to x^*$ for some $x^* $ in $\mathbb{X}$. Then $T$ is continuous at $x^*$.
\end{theorem}
\begin{proof}
	Suppose $T$ is not  continuous at $x^*\in \mathbb{X}.$ So we must have $\displaystyle \lim_{n\to \infty }T(x_n)\neq Tx^*$ i.e., $d(x^*,Tx^*) >0$. Now, 
	$$d(Tx^*,T^nx) +d(T^nx,T^{n+1}x) +d(T^{n+1}x,Tx^*) \leq k M\big(d(x',y')+d(y', z') +d(z', x')\big )$$ where $x',y',z'\in \{ x^*, T^{n-1}x, T^{n}x,T^{n+1}x, Tx^* \}$. Taking $i\to \infty$ in the both sides of the above inequality, one can derive that $M\big(d(x',y')+d(y', z') +d(z', x') \big)= 2d(Tx^*,x^*)$ and hence we obtain 
	$$d(Tx^*,x^*)\leq k d(Tx^*,x^*)$$ which is a contradiction. Hence $T$ must be continuous at $x^*$.
\end{proof}

\section{\bf Completeness}
Now we are interested in characterizing the completeness property of the underlying spaces through the fixed point results of map contracting perimeters of triangles. There is a vast literature analyzing the completeness property of the underlying spaces in terms of fixed point results, see \cite{con,park, sub,suz}. We observe that like Banach Contraction, the map contracting perimeters of triangles can not characterize the completeness property of the underlying spaces.  Suzuki and Takahasi \cite{suz} presented an example of an incomplete metric space in which every continuous map has a fixed point.  Since every map contracting perimeters of triangles in a metric space is continuous, the result follows immediately from \cite{suz}. On the other hand, we can characterize the completeness property of the underlying spaces in terms of fixed point of the weak contraction map on the perimeters of the triangles. In this direction, we present the following theorem.
\begin{theorem}
	Let $(\mathbb{X},d)$ be a metric space. Every mapping $T:\mathbb{X} \to \mathbb{X}$ satisfying the following conditions:
	\begin{enumerate}
		\item[(i)]  $T^2(x) \neq x$ for all $x\in \mathbb{X}$ with $Tx\neq x$;
		\item[(ii)] $d(Tx,Ty) +d(Ty,Tz) +d(Tz,Tx) \leq k M\big(d(x',y')+d(y', z') +d(z', x') \big)$\\ where $k\in (0,1)$ and $M\big(d(x', y') +d(y', z') +d(z', x')\big)$ is defined as\\
		$\max\{d(x', y') +d(y', z') +d(z', x'): x',y',z'\in \{x,y,z,Tx,Ty,Tz\}~
		\mbox{and}~ x'\neq y'\neq z'\};$
	\end{enumerate}
	has a fixed point, then $(\mathbb{X},d)$ must be complete.
\end{theorem}
\begin{proof}
	We can prove this theorem by following the idea of Theorem $1$ of \cite{sub}  given by P. V. Subrahmanyam. Let $(\mathbb{X},d)$ be an incomplete metric space. So there exists a non-convergent Cauchy sequence $(x_n)$ in $\mathbb{X}$. Let $A=(x_n)$ where every term is distinct.  Therefore, as given in \cite{sub}, for every $x\in \mathbb{X}\setminus A$ there
	exists $N_x$ such that 
	$$ d(x_m,x_{N_x})<\frac{k}{2} d(x,x_i), ~~i=1,2,3,...;\forall m\geq N_x;$$ 
	and for every $x_n\in A$ there exists $n'(n)>n$ such that 
	$$ d(x_m,x_{n'})<\frac{k}{2} d(x_n,x_{n'})~~\forall m\geq n'.$$
	Now, $T:\mathbb{X}\to \mathbb{X}$ is defined as 
	\begin{align}
	Tx =
	\begin{cases}
	x_{N_x}, & x\notin A\nonumber ;\\
	x_{n'}, &x=x_n\in A.\nonumber
	\end{cases}
	\end{align}

	Clearly, for any $x\in \mathbb{X}, T^2(x) \neq x$, i.e, $T$ satisfies $(i)$.
	Now for any $x,y\in \mathbb{X}$, $d(Tx,Ty)=d(x_n,x_m)<\frac{k}{2} d(y,A\setminus \{y\}), ~\text {if} ~n\geq m$ or $d(Tx,Ty) =d(x_n,x_m)<\frac{k}{2} d(x,A\setminus \{x\}), ~\text{if} ~n< m$. Therefore, for any three pairwise distinct  $x,y,z\in \mathbb{X}$, let $Tx=x_n,Ty=x_m$ and $Tz=x_p$.  Without loss of generality, let $n\geq m\geq p$.  Note that $z\neq Ty$. Then, we have
	\begin{align*}
	d(Tx,Ty) +d(Ty,Tz) +d(Tz,Tx) & < \frac{k}{2} \{d(y,Ty) +d(z,Ty)+d(z,Ty ) \}\\
	& <  \frac{k}{2} \{d(y,Ty) +d(z,Ty)+d(z,Ty ) \}\\
	& +  \frac{k}{2} d(y,Ty)+ k d(y,z)\\
	& <  k\{d(y,Ty) +d(Ty,z)+d(z,y ) \} \\
	& \leq  k M\big(d(x',y')+d(y', z') +d(z', x')\big)  
	\end{align*}
	where $x',y',z'\in \{x,y,z,x_n,x_m,x_p \}$.
	This shows that $T$ satisfies $(ii)$. Note that $T$ does not have any fixed point. This completes our proof.
	
\end{proof}
In view of the above result, we conclude this article with the following theorem.
\begin{theorem}
	A metric space $(\mathbb{X},d)$ is complete if and only if every weak contraction on the perimeters of triangles in $\mathbb{X}$, satisfing $T^2x\neq x$ for any $x\in \mathbb{X}$ with $Tx\neq x$, has fixed point in $\mathbb{X}$. 
\end{theorem}


\end{document}